\documentclass[12 pt]{article}%
\usepackage{amsmath, amsfonts, amsthm, color,latexsym}
\usepackage{amsmath}
\usepackage{amsfonts}
\usepackage{amssymb}
\usepackage{color, soul}
\usepackage[all]{xy}
\usepackage{graphicx}%
\providecommand{\U}[1]{\protect\rule{.1in}{.1in}}
\allowdisplaybreaks[4]
\newtheorem{theorem}{Theorem}[section]
\newtheorem{proposition}[theorem]{Proposition}
\newtheorem{corollary}[theorem]{Corollary}
\newtheorem{example}[theorem]{Example}

\newtheorem{final remark}[theorem]{Final Remark}
\newtheorem{definition}[theorem]{Definition}
\allowdisplaybreaks[4]

\begin{document}

\title{Injective polynomial ideals and the domination property}
\author{Geraldo Botelho\thanks{Supported by CNPq Grant
304262/2018-8 and Fapemig Grant PPM-00450-17.}\,\, and  Leodan A. Torres\thanks{Supported by a CNPq scholarship\newline 2010 Mathematics Subject Classification: 47L22, 46G25, 47B10, 47B33..
}}
\date{}
\maketitle

\begin{abstract} After sketching the basic theory of injective ideals of homogeneous polynomials, we characterize injective polynomial ideals by means of a domination property and applications of this characterization to some classical operator ideals and to composition polynomial ideals are provided.
\end{abstract}

\section{Introduction and background}

As a consequence of the successful theory of ideals of linear operators (operator ideals), ideals of continuous homogeneous polynomials between Banach spaces have been intensively studied since Pietsch \cite{pietsch} introduced the concept of ideals of multilinear operators. Contrary to the case of surjective polynomial ideals, which were thoroughly investigated in \cite{berrios}, injective polynomial ideals have not been studied yet. The aim of this note is to fill this gap.

In Section 2 we outline the basic theory of injective polynomial ideals. We give the definition, provide illustrative examples, characterize injective polynomial ideals by means of the injective hull and establish the properties of a hull procedure. The main results of the paper appear in Section 3. Inspired by the fact that injective operator ideals are characterized by a domination property, we investigate the situation in the polynomial case. First we announce that, by means of counterexample that will appear at the end of  the paper, the polynomial analogue of the linear domination property does not characterize injective polynomial ideals. One of our main results is the identification of a related domination property that characterizes injective polynomial ideals. A first application of this result is the characterization of injective composition polynomial ideals, a class that encompasses some classical polynomial ideals. Several applications follow, involving the ideals of finite rank, approximable, compact and weakly operators/polynomials, the polynomial dual of an operator ideal and the ideals of $p$-compact and Cohen strongly $p$-summing linear operators.

For Banach spaces $E$ and $F$, $B_E$ denotes the closed unit ball of $E$, $E^*$ denotes the topological dual of $E$, ${\cal L}(E;F)$ is the space of bounded linear operators from $E$ to $F$ endowed with the usual sup norm, ${\cal P}(^mE;F)$ is the space of continuous $m$-homogeneous polynomials from $E$ to $F$. A metric injection is a linear operator $j \colon E \longrightarrow F$ such that $\|j(x)\| = \|x\|$ for every $x \in E$. The metric injection
$$I_E\colon E \longrightarrow \ell_{\infty}(B_{E^*})~,~I_{E}(x)=\left(\varphi(x)\right)_{\varphi\in B_{E^*}},$$
is called the canonical metric injection.

Operator ideals will be taken in the sense of Pietsch \cite{df, djp, livropietsch}, ideals of homogeneous polynomials (polynomial ideals) in the sense of \cite{fg, fh}  and polynomial hyper-ideals in the sense of \cite{ewertonlaa}. For the sake of the reader, we recall these concepts next.

\begin{definition}\rm Let ${\cal Q}$ be a subclass of the class of homogeneous polynomials between Banach spaces such that, for every $m$ and any Banach spaces $E$ and $F$, the component
$${\cal Q}(^mE;F) := {\cal P}(^mE;F) \cap {\cal Q} $$
is a linear subspace of ${\cal P}(^mE;F)$ containing the polynomials of finite type. The class $\cal Q$ is said to be:\\
(a) A {\it polynomial ideal} if $t \circ P \circ u \in {\cal Q}(^mE;H)$ whenever $t \in {\cal L}(G;H)$, $P \in {\cal Q}(^mF;G)$ and $u \in {\cal L}(E;F)$.\\
(b) A {\it polynomial hyper-ideal} if $t \circ P \circ Q \in {\cal Q}(^{mn}E;H)$ whenever $t \in {\cal L}(G;H)$, $P \in {\cal Q}(^mF;G)$ and $Q \in {\cal P}(^nE;F)$.

Suppose that there is a function $\|\cdot\|_{\cal Q} \colon {\cal Q} \longrightarrow \mathbb{R}$ whose restriction to each component ${\cal Q}(^mE;F)$ is a norm such that $\|\lambda \in \mathbb{K} \mapsto \lambda^m\|_{\cal Q} = 1$ for every $m$. $({\cal Q}, \|\cdot\|_{\cal Q})$ is said to be:\\
(a') A {\it normed polynomial ideal} if, in (a), $\|t \circ P \circ u\|_{\cal Q} \leq \|t\|\cdot\|P\|_{\cal Q}\cdot \|u\|$.\\
(b') A {\it normed polynomial hyper-ideal ideal} if, in (b), $\|t \circ P \circ Q\|_{\cal Q} \leq \|t\|\cdot\|P\|_{\cal Q}\cdot \|Q\|^m$.

If each component $({\cal Q}(^m E;F), \|\cdot\|_{\cal Q})$ is a Banach space, then $({\cal Q}, \|\cdot\|_{\cal Q})$ is said to be a {\it Banach polynomial ideal} or a {\it Banach polynomial hyper-ideal}. If the norm is the usual sup norm, we speak of a {\it closed polynomial ideal} or a {\it closed polynomial hyper-ideal}.
\end{definition}

The $m^{\rm th}$-component of a polynomial ideal $\cal Q$, defined as
$${\cal Q}_m = \textstyle\bigcup\limits_{E,F} {\cal Q}(^mE;F), $$
is called a (normed, Banach, closed) ideal of $m$-homogeneous polynomials. Of course, its linear component ${\cal Q}_1$ is an operator ideal.

Just to mention a few illustrative examples, the class of nuclear polynomials is a Banach polynomial ideal that fails to be a hyper-ideal and the classes of compact and weakly compact polynomials are closed hyper-ideals.

For the basic theory of homogeneous polynomials we refer to \cite{livrodineen, mujica}. 
\section{Injective polynomial ideals}

Like in the linear case, a polynomial ideal is injective if the containment of a polynomial in the class depends on the norm of the target space rather than on the space itself.

\begin{definition}\rm A polynomial ideal $\cal Q$ is said to be {\it injective} if $P\in \mathcal{Q}(^m E;F)$ whenever $P\in\mathcal{P}(^m E;F)$ and $j \colon F \longrightarrow G$ is a metric injection such that $j\circ P\in \mathcal{Q}(^m E;G)$.

A normed polynomial ideal $\left(\mathcal{Q},\|\cdot\|_{\mathcal{Q}}\right)$ is {\it injective} if $\cal Q$ is an injective polynomial ideal and, in the situation above, $\|P\|_{\cal Q} = \|j\circ P\|_{\cal Q}$. 
\end{definition}

\begin{example}\rm It is easy to check that the ideals ${\cal P}_{\cal F}$ of finite rank polynomials (the range of the polynomial generates a is finite-dimensional subspace of the target space), ${\cal P}_{\cal K}$ of compact polynomials (bounded sets are sent to relatively compact sets) and ${\cal P}_{\cal W}$ of weakly compact polynomials (bounded sets are sent to relatively weakly compact sets) are injective. In Corollary \ref{approxima} we shall prove that the ideal of approximable polynomials, the ones that can be approximated, in the usual sup norm, by finite rank polynomials, is not injective. It is obvious that all ideals of polynomials of summing type (absolutely summing, dominated, strongly summing, multiple summing, etc) are injective. Corollary \ref{corol} provides plenty of injective and non-injective polynomial ideals.
\end{example}

We aim to characterize injective polynomial ideals by the coincidence with its injective hull.

\begin{proposition}\label{inje}{\rm (a)} Let $\mathcal{Q}$ be a polynomial ideal (polynomial hyper-ideal, respectively). Then there exists a unique smallest injective polynomial ideal (hyper-ideal, respectively) $\mathcal{Q}^{inj}$ containing $\mathcal{Q}$. For 
$P\in\mathcal{P}(^mE;F)$,
$$P\in\mathcal{Q}^{inj}(^mE;F)\Longleftrightarrow I_F\circ P\in\mathcal{Q}(^mE; \ell_\infty(B_{F^*}),$$
where $I_F$ is the canonical metric injection.\\
{\rm (b)} Let $\left(\mathcal{Q},\|\cdot\|_{\mathcal{Q}}\right)$ be normed (Banach)  polynomial ideal (polynomial hyper-ideal). Then there exists a unique smallest normed (Banach) injective polynomial ideal (polynomial hyper-ideal) $(\mathcal{Q}^{inj}, \|\cdot\|_{\mathcal{Q}^{inj}})$ containing $\mathcal{Q}$ and such that $\|\cdot\|_{\mathcal{Q}^{inj}} \leq \|\cdot\|_{\cal Q}$. For 
$P\in\mathcal{P}(^mE;F)$,
\begin{equation}\label{eqdef}P\in\mathcal{Q}^{inj}(^mE;F)\Longleftrightarrow I_F\circ P\in\mathcal{Q}(^mE; \ell_\infty(B_{F^*}) {\rm ~~ and~~}\|P\|_{\mathcal{Q}^{inj}}:=\|I_F\circ P\|_{\mathcal{Q}}.
\end{equation}
\end{proposition}

The ideal $\mathcal{Q}^{inj}$ (normed, Banach ideal $(\mathcal{Q}^{inj}, \|\cdot\|_{\mathcal{Q}^{inj}})$) is called the {\it injective hull} of the ideal $\cal Q$ (normed, Banach ideal $\left(\mathcal{Q},\|\cdot\|_{\mathcal{Q}}\right)$).

\begin{proof} $\mathcal{Q}^{inj}$ and $\|\cdot\|_{\mathcal{Q}^{inj}}$ are defined according to (\ref{eqdef}). We check only the hyper-ideal property, the other statements follow from standard arguments. Let $Q\in \mathcal{P}(^nE;F)$, $P\in \mathcal{Q}^{inj}(^m F;G)$ and $t\in\mathcal{L}(G;H)$ be given. By the definition of  $\mathcal{Q}^{inj}(^m F;G)$, $I_G\circ P\in\mathcal{Q}(^mF; \ell_\infty(B_{G^*})$. Since $I_G $ is a metric injection, an application of the metric extension property of   $\ell_\infty(B_{H^{*}})$ \cite[Proposition C.3.2.1]{livropietsch} to the operator $I_H\circ t$ gives rise to an operator $s\in\mathcal{L}(\ell_\infty(B_{G^{*}});\ell_\infty(B_{H^{*}}))$ such that $I_H\circ t=s\circ I_G$ and $\|s\|=\|I_H\circ t\|$.
$$
\xymatrix{
E \ar[r]^Q & F\ar[r]^P & G \ar[r]^t \ar[d]_{I_G} & H\ar[r]^{I_H} & \ell_\infty(B_{H^{*}})\\
 &  & \ell_\infty(B_{G^{*}})\ar[rru]^s &  &
}
$$
Therefore,
$$I_H\circ t\circ P\circ Q=(s\circ I_G)\circ P\circ Q=s\circ(I_G\circ P)\circ Q \in \mathcal{Q}(^{mn}E; \ell_\infty(B_{H^*})),$$
that is, $t\circ P\circ Q\in \mathcal{Q}^{inj}(^{mn} E;H)$ and
\begin{align*}
\|t\circ P\circ Q\|_{\mathcal{Q}^{inj}} & = \|I_H\circ t\circ P\circ Q\|_{\mathcal{Q}}= \|s\circ I_G\circ P\circ Q\|_{\mathcal{Q}}\leq \|s\| \cdot\|I_G\circ P\|_{\mathcal{Q}}\cdot \|Q\|^m \\
&= \|I_H\circ t\| \cdot\|I_G\circ P\|_{\mathcal{Q}} \cdot\|Q\|^m\leq \|t\| \cdot\|P\|_{\mathcal{Q}^{inj}} \cdot\|Q\|^m.
\end{align*}
\end{proof}

\begin{corollary}\label{corinje} {\rm (a)} A polynomial ideal (hyper-ideal) $\cal Q$ is injective if and only if ${\cal Q} = \mathcal{Q}^{inj}$.\\
{\rm (b)} A normed (Banach) polynomial ideal (hyper-ideal) is injective if and only if ${\cal Q} = \mathcal{Q}^{inj}$ isometrically.
\end{corollary}

Next we check that the correspondence ${\cal Q} \mapsto {\cal Q}^{inj}$ is a hull procedure in the sense of \cite[8.1.2]{livropietsch}. We state only the case of normed/Banach polynomial ideals/hyper-ideals. The non-normed case is a straightforward consequence.

\begin{proposition}\label{hull} Let $(\mathcal{Q},\|\cdot\|_{\mathcal{Q}})$ and $(\mathcal{R},\|\cdot\|_{\mathcal{R}})$ be normed (Banach) polynomial ideals (polynomial hyper-ideals). Then:\\
{\rm (a)} $\left(\mathcal{Q}^{inj},\|\cdot\|_{\mathcal{Q}^{inj}}\right)$ is a normed (Banach) polynomial ideal (polynomial hyper-ideal).\\
{\rm (b)} If $\mathcal{Q}\subseteq \mathcal{R}$ and $\|\cdot\|_{\mathcal{R}}\leq\|\cdot\|_{\mathcal{Q}}$, then $\mathcal{Q}^{inj}\subseteq \mathcal{R}^{inj}$ and $\|\cdot\|_{\mathcal{R}^{inj}}\leq\|\cdot\|_{\mathcal{Q}^{inj}}$.\\
{\rm (c)} $(\mathcal{Q}^{inj})^{inj}=\mathcal{Q}^{inj}$ and $\|\cdot\|_{\left(\mathcal{Q}^{inj}\right)^{inj}}=\|\cdot\|_{\mathcal{Q}^{inj}}$.\\
{\rm (d)} $\mathcal{Q}\subseteq\mathcal{Q}^{inj}$ and $\|\cdot\|_{\mathcal{Q}^{inj}} \leq\|\cdot\|_{\mathcal{Q}}$.
\end{proposition}

\begin{proof} (a) and (d) follow from Proposition \ref{inje} and (c) follows from Corollary \ref{corinje}. To prove (b), let $P\in\mathcal{Q}^{inj}(^m E;F)$ be given. Thus $I_F\circ P\in \mathcal{Q}(^m E;\ell_{\infty}(B_{F^{*}}))\subseteq \mathcal{R}(^m E;\ell_{\infty}(B_{F^{*}}))$, what gives $P\in\mathcal{R}^{inj}(^m E;F)$. Moreover,
$$\|P\|_{\mathcal{R}^{inj}}=\|I_F\circ P\|_{\mathcal{R}}\leq\|I_F\circ P\|_{\mathcal{Q}}=\|P\|_{\mathcal{Q}^{inj}}.$$
\end{proof}

\section{The domination property}

Injective operator ideals are characterized by the following domination property:

\begin{proposition}\label{linear}{\rm \cite[Exercise 9.10(b)]{df}, \cite[Lemma 3.1]{pacific}} An operator ideal $\cal I$ is injective if and only if given operators $u\in \mathcal{I}(E;F)$ and $v\in \mathcal{L}(E;G)$ such that
	$$\left\|v(x)\right\| \leq C\cdot \left\|u(x)\right\| $$
	for every $x\in E$ and some constant $C \geq 0$ (eventually depending on $E$, $F$, $G$, $u$, $v$), then $v\in \mathcal{I}(E;G)$.
\end{proposition}

Transposing the linear domination property above literally to the polynomial case, we end up with the following:

\begin{definition}\rm A polynomial ideal $\cal Q$ is said to have the {\it weak domination property} if given polynomials $P\in \mathcal{Q}(^m E;F)$ and $Q\in \mathcal{P}(^m E;G)$ such that
	$$\left\|Q(x)\right\| \leq C\cdot \left\|P(x)\right\| $$
	for every $x\in E$ and some constant $C \geq 0$ (eventually depending on $E$, $F$, $G$, $P$, $Q$, $m$), then $Q\in \mathcal{Q}(^m E;G)$.
\end{definition}

Given a polynomial $P \in {\cal P}(^mE;F)$ and a metric injection $j \colon F \longrightarrow G$, we have
$$\|P(x)\| = \|j(P(x))\| = \|(j \circ P)(x)\| $$
for every $x \in E$. So, the weak domination property is sufficient for a polynomial ideal to be injective: Every polynomial ideal with the weak domination property is injective.

In the linear case, the proof that every injective operator ideal has the domination property depends heavily on the linearity of the underlying operators, so it is not expected that every injective polynomial ideal has the weak domination property. Indeed, in Example \ref{exfinal} we shall give an example of an injective polynomial ideal failing the weak domination property, which establishes that this property does not characterize injective polynomial ideals. This poses two questions: Can injective polynomial ideals be characterized by some related domination property? If yes, is this characterization useful? Next we answer these two questions affirmatively. 

\begin{definition}\rm A polynomial ideal $\cal Q$ is said to have the {\it strong domination property} if given polynomials $P\in \mathcal{Q}(^m E;F)$ and $Q\in \mathcal{P}(^m E;G)$ such that
	$$\left\|\sum_{i=1}^k\lambda_i Q(x_i)\right\| \leq C\cdot \left\|\sum_{i=1}^k\lambda_i P(x_i)\right\| $$
	for all $k \in \mathbb{N}$, $x_1, \ldots, x_k \in E$, $\lambda_1, \ldots, \lambda_k\in \mathbb{K}$ and some constant $C \geq 0$ (eventually depending on $E$, $F$, $G$, $P$, $Q$, $m$), then $Q\in \mathcal{Q}(^m E;G)$.
\end{definition}

\begin{theorem}\label{thdp} A polynomial ideal is injective if and only if it has the strong domination property.
\end{theorem}

\begin{proof} Suppose that $\cal Q$ is an injective polynomial ideal and let  $P\in\mathcal{Q}(^m E;F)$ and $Q\in\mathcal{P}(^m E;G)$ be such that
	\begin{equation}
	\left\|\sum_{i=1}^k\lambda_i Q(x_i)\right\| \leq C\cdot \left\|\sum_{i=1}^k\lambda_i P(x_i)\right\|
	\label{4aast}
	\end{equation}
	for all $k \in \mathbb{N}$, $x_1, \ldots, x_k \in E$, $\lambda_1, \ldots, \lambda_k\in \mathbb{K} $ and some constant $C$. Let us see that the operator
 $$W\colon {\rm span}\{P(E)\}\subseteq F\longrightarrow G~,~ W\left(\sum_{i=1}^k\lambda_i P(x_i)\right)=\sum_{i=1}^k\lambda_i Q(x_i),$$
is well defined: indeed,
\begin{align*}\sum\limits_{i=1}^k\lambda_i P(x_i)=\sum\limits_{j=1}^l\alpha_j P(x_j) & \Longrightarrow \left\|\sum_{i=1}^k\lambda_i P(x_i)-\sum_{j=1}^l\alpha_j P(x_j)\right\|=0\\
 & \stackrel{(\ref{4aast})}{\Longrightarrow} \left\|\sum_{i=1}^k\lambda_i Q(x_i)-\sum_{j=1}^l\alpha_j Q(x_j) \right\|=0\\
 &\Longrightarrow \sum_{i=1}^k\lambda_i Q(x_i)=\sum_{j=1}^l\alpha_j Q(x_j).
 \end{align*}
The linearity of $W$ é clear and its continuity follows from
	$$\left\|W\left(\sum_{i=1}^k\lambda_i P(x_i) \right)\right\|=\left\|  \sum_{i=1}^k\lambda_i Q(x_i) \right\|\stackrel{(\ref{4aast})}{\leq}C \cdot\left\|  \sum_{i=1}^k\lambda_i P(x_i) \right\|.$$
	Then there exists a (unique) bounded linear operator $W_1\colon \overline{{\rm span}\{P(E)\}}\subseteq F\longrightarrow G$ such that $W_1|_{{\rm span}\{P(E)\}}=W$. Denoting by $i \colon \overline{{\rm span}\{P(E)\}}\longrightarrow F$ the formal inclusion operator and by $P_0 \colon E \longrightarrow {\rm span}\{P(E)\}$ the obvious polynomial, we have the diagram
	\begin{equation*}
	\begin{gathered}
	\xymatrix@C02pt@R18pt{
		E\ar[dd]_*{Q} \ar[rrrrrrrrrr]^*{P_0} & & & & &  & & & & & \overline{{\rm span}\{P(E)\}}\ar@{-->}[rrrrrrrrrrdd]_{J_G\circ W_1}\ar[lllllllllldd]_{W_1} \ar[rrrrrrrrrr]^*{i}&&&&&&&&&&F\ar@{-->}[dd]^{\stackrel{\sim}{W_1}}\\
		&  &  &  &  & & & & &  &   &\\
		G \ar[rrrrrrrrrrrrrrrrrrrr]_{J_G}& & & & & & & &&&&&&&&& & &  & &\ell_\infty(B_{G^{*}})
	}
	\end{gathered}
	\end{equation*}
	As $i$ is a metric injection, from the metric approximation property of $\ell_\infty(B_{G^{*}})$ there exists an operator  $\stackrel{\sim}{W_1}\in\mathcal{L}(F;\ell_\infty(B_{G^{*}}))$ such that $J_G\circ W_1=\,\stackrel{\sim}{W_1}\circ\ i$ and $\|\stackrel{\sim}{W_1}\|=\|J_G\circ W_1\|$. From
	$$(W_1\circ P_0)(x)=W_1(P(x))=W(P(x))=Q(x) {\rm ~for~every~}x\in E,$$
that is $Q = W_1 \circ P_0$, we conclude that
$$J_G\circ Q= J_G \circ W_1\circ P_0=\,\stackrel{\sim}{W_1}\circ \ i\circ P_0=\,\stackrel{\sim}{W_1}\circ~ P.$$ Since $P\in \mathcal{Q}(^mE;F)$, the ideal property of $\cal Q$ gives $J_G\circ Q\in \mathcal{Q}(^mE;\ell_\infty(B_{G^{*}}))$. The injectivity of $\mathcal{Q}$ and the fact that $J_G$ is a metric injection give $Q\in\mathcal{Q}(^mE;G)$, showing that $\cal Q$ has the strong domination property.

Conversely, suppose that $\cal Q$ is a polynomial ideal with the strong domination property. Given $P \in {\cal P}(^mE;F)$ and a metric injection $j \colon F \longrightarrow G$ such that $(j \circ P)\in {\cal Q}(^mE;F)$, we have
$$\left\|\sum_{i=1}^k \lambda_iP(x_i) \right\| = \left\|j\left(\sum_{i=1}^k \lambda_i P(x_i)\right) \right\|= \left\|\sum_{i=1}^k \lambda_i(j\circ P)(x_i) \right\| $$
for all $k \in \mathbb{N}$, $x_1, \ldots, x_k \in E$ and $\lambda_1, \ldots, \lambda_k\in \mathbb{K}$. The strong domination property of $\cal Q$ gives $P \in {\cal Q}(^mE;F)$, proving that $\cal Q$ is injective.
\end{proof}

Now we apply the characterization above to establish a quite useful formula regarding composition polynomial ideals, whose definition goes back to Pietsch \cite{pietsch} and we recall now: Given an operator ideal $\cal I$, a polynomial $P \in {\cal P}(^mE;F)$ belongs to ${\cal I} \circ {\cal P}(^mE;F)$ if there exist a Banach space $G$, a polynomial $Q \in {\cal P}(^mE;G)$ and an operator $u \in {\cal I}(G;F)$ such that $P = u \circ Q$. It is well known that ${\cal I} \circ {\cal P}$ is a polynomial hyper-ideal.

\begin{theorem}\label{comp} For every operator ideal $\cal I$,
$${\cal I}^{inj} \circ {\cal P} = ({\cal I} \circ {\cal P})^{inj}. $$
In particular, the polynomial hyper-ideal ${\cal I}^{inj} \circ {\cal P}$ is injective.
\end{theorem}

\begin{proof} Given $P \in {\cal I}^{inj}(^mE;F)$, $P = u \circ Q$ for some Banach space $G$, $Q \in {\cal P}(^mE;G)$ and $u \in {\cal I}(G;F)$. Then the factorization
$I_F \circ P = I_F \circ u  \circ Q$ with $I_F \circ u  \in {\cal I}(G; \ell_\infty(B_{G^*}))$ shows that $I_F \circ P$ belongs to ${\cal I} \circ {\cal P}$. This proves that
\begin{equation}\label{eqqe}{\cal I}^{inj} \circ {\cal P} \subseteq ({\cal I} \circ {\cal P})^{inj}. \end{equation}

Let us prove that ${\cal I}^{inj} \circ {\cal P}$ is injective. Let $P \in {\cal I}^{inj} \circ {\cal P}(^mE;F)$ and $Q \in {\cal P}(^mE;G)$ be such that
	$$\left\|\sum_{i=1}^k\lambda_i Q(x_i)\right\| \leq C\cdot \left\|\sum_{i=1}^k\lambda_i P(x_i)\right\| $$
	for all $k \in \mathbb{N}$, $x_1, \ldots, x_k \in E$, $\lambda_1, \ldots, \lambda_k\in \mathbb{K}$ and some constant $C \geq 0$. Call $P_L$ and $Q_L$ the linearizations of $P$ and $Q$ on the (completed) projective symmetric tensor product, that is $P_L \colon \widehat{\otimes}_{\pi}^{m,s}E \longrightarrow F$ and $Q_L \colon \widehat{\otimes}_{\pi}^{m,s}E \longrightarrow G$ are bounded linear operators such that
$$P_L (\otimes^m x) = P(x) {\rm ~~and~~} Q_L (\otimes^m x) = Q(x) $$
 for every $x \in E$ (see \cite{f}). Given $z = \sum\limits_{i=1}^k \lambda_i \otimes^m x_i$ in the (incomplete) symmetric tensor product ${\otimes}_{\pi}^{m,s}E$, we have
 $$\|Q_L(z)\| = \left\|\sum_{i=1}^k\lambda_i Q(x_i)\right\| \leq C\cdot \left\|\sum_{i=1}^k\lambda_i P(x_i)\right\|= C \cdot \|P_L(z)\|.$$
 The continuity of $P$, of $Q$ and of the norm give that $\|Q_L(z)\| \leq C \cdot \|P_L(z)\|$ for every $z \in \widehat{\otimes}_{\pi}^{m,s}E$.
Since $P \in {\cal I}^{inj} \circ {\cal P}(^mE;F)$, we know from \cite[Proposition 3.2]{Geraldo.D.P} that $P_L \in {\cal I}^{inj}\left(\widehat{\otimes}_{\pi}^{m,s}E;F) \right)$. Now the injectivity of ${\cal I}^{inj}$ and Proposition \ref{linear} give $Q_L \in {\cal I}^{inj}\left(\widehat{\otimes}_{\pi}^{m,s}G;F) \right)$. Calling on \cite[Proposition 3.2]{Geraldo.D.P} once again we get $Q \in {\cal I}^{inj} \circ {\cal P}(^mG;F)$, proving that ${\cal I}^{inj} \circ {\cal P}$ has the strong domination property, hence it is injective by Theorem \ref{thdp}. Combining this with (\ref{eqqe}), with ${\cal I}\subseteq {\cal I}^{inj}$ and with Proposition \ref{hull}(b), we get
$$\left({\cal I}^{inj} \circ {\cal P}\right)^{inj} = {\cal I}^{inj} \circ {\cal P}  \subseteq ({\cal I} \circ {\cal P})^{inj} \subseteq \left({\cal I}^{inj} \circ {\cal P}\right)^{inj},$$
which gives the desired formula. The second assertion follows from Proposition \ref{inje}.
\end{proof}

\begin{corollary}\label{corol} The following are equivalent for an operator ideal $\cal I$:\\
{\rm (a)} $\cal I$ is an injective operator ideal.\\
{\rm (b)} ${\cal I} \circ {\cal P}$ is an injective polynomial hyper-ideal.\\
{\rm (c)} $({\cal I} \circ {\cal P})_m$ is an injective ideal of $m$-homogeneous polynomials for some $m \in \mathbb{N}$.
\end{corollary}

\begin{proof} (a) $\Longrightarrow$ (b) follows from Theorem \ref{comp} and Corollary \ref{corinje}, (b) $\Longrightarrow$ (c) is obvious and (c) $\Longrightarrow$ (a) follows from \cite[Lemma 3.4]{Geraldo.D.P}.
\end{proof}

Let $\cal F$, ${\cal K}$ and ${\cal W}$ denote the injective ideals of finite rank, compact and weakly compact polynomials. Since ${\cal P}_{\cal F} = {\cal F} \circ {\cal P}$ \cite[Lemma 2.1]{jmaaleticia}, ${\cal P}_{\cal K} = {\cal K} \circ {\cal P}$ \cite[Proposition 4.1]{ryan.tese} and ${\cal P}_{\cal W} = {\cal W} \circ {\cal P}$ \cite[Proposition 4.1]{ryan.tese}, the corollary above gives another proof that the polynomial ideals of finite rank, compact and weakly compact polynomials are injective.

Now we compute the injective hull of the closed polynomial ideal ${\cal P}_{\cal A}$ of polynomials that can be approximated, in the usual sup norm, by polynomials of finite rank, that is, ${\cal P}_{\cal A} = \overline{{\cal P}_{\cal F}}$.

\begin{corollary}\label{approxima} $({\cal P}_{\cal A})^{inj} = {\cal P}_{\cal K}$.
\end{corollary}

\begin{proof} Denoting by ${\cal A} = \overline{\cal F}$ the ideal of operators that can be approximated, in the usual sup norm, by finite rank operators, since ${\cal A}^{inj} = {\cal K}$ \cite[Proposition 19.2.3]{jarchow} and ${\mathcal P}_{\cal
A} = {\cal A}\circ {\mathcal P}$ \cite[Theorem 2.2]{jmaaleticia}, the proof follows from Theorem \ref{comp}:
$$ ({\mathcal P}_{\cal
A})^{inj} = ({\cal A}\circ {\mathcal P})^{inj}= {\cal A}^{inj}\circ {\mathcal P} = {\cal K} \circ {\cal P} =
\mathcal{P}_{\cal K}.$$
\end{proof}

The next application concerns the polynomial dual ${\cal I}^{{\cal P}{-dual}}$ of a given operator ideal $\cal I$ defined in \cite{Geraldo.C.M} as
$${\cal I}^{{\cal P}{-dual}}(^mE;F) = \{P \in {\cal P}(^mE;F) : P^* \in {\cal I}(F^*;{\cal P}(^mE))\}, $$
where $P^*$ is the Aron-Schottenloher adjoint of $P$, that is, $P^*(\varphi)(x) = \varphi(P(x))$ \cite{aron}.

An operator ideal is said to be symmetric if ${\cal I} = {\cal I}^{dual}$.

\begin{corollary} A symmetric operator ideal is injective if and only if its polynomial dual is an injective polynomial ideal.
\end{corollary}

\begin{proof} Since ${\cal I}^{{\cal P}{-dual}} = {\cal I}^{dual} \circ {\cal P}$ for every operator ideal $\cal I$ \cite[Theorem 2.2]{Geraldo.C.M}, the result follows from Corollary \ref{corol}.
\end{proof}

Now we characterize the polynomial duals of the ideal $\mathfrak{K}_p$ of $p$-compact operators (see \cite{pietschpams}) and of the ideal ${\cal D}_p$ of Cohen strongly $p$-summing operators (see \cite{pacific, cohen}). $\mathfrak{N}_p$ stands for the ideal of $p$-nuclear operators and $\mathfrak{I}_p$ for the ideal of $p$-integral operators (see \cite{livropietsch}).

\begin{corollary} $\mathfrak{K}_p^{{\cal P}-dual} = (\mathfrak{N}_p \circ {\cal P})^{inj}$ and ${\cal D}_p^{{\cal P}-dual} = (\mathfrak{I}_{p^*} \circ {\cal P})^{inj}$.
\end{corollary}

\begin{proof} In
$$\mathfrak{K}_p^{{\cal P}-dual} = \mathfrak{K}_p^{dual} \circ {\cal P}= \mathfrak{N}_p^{inj} \circ {\cal P}=   (\mathfrak{N}_p \circ {\cal P})^{inj}, $$
the first equality follows from \cite[Theorem 2.2]{Geraldo.C.M}, the second from \cite[Theorem 6]{pietschpams} and the third from Theorem \ref{comp}; and in
$${\cal D}_p^{{\cal P}-dual} = {\cal D}_p^{dual} \circ {\cal P} = \Pi_{p^*} \circ {\cal P}  = \mathfrak{I}_{p^*}^{inj} \circ {\cal P}= (\mathfrak{I}_{p^*} \circ {\cal P})^{inj}, $$
$\Pi_{p^*}$ denotes the ideal of absolutely $p^*$-summing operators, the first equality follows from \cite[Theorem 2.2]{Geraldo.C.M}, the second from \cite{cohen}, the third from \cite[Theorem 19.2.7]{livropietsch} and the fourth from Theorem \ref{comp}.
\end{proof}

Our final application is the promised example of an injective polynomial ideal failing the weak domination property, which establishes, in particular, that the weak and the strong domination properties are not equivalent.

\begin{example}\label{exfinal}\rm Consider the injective closed operator ideal $\cal CC$ of completely continuous operators (weakly convergent sequences are sent to norm convergent sequences) and the continuous 2-homogeneous polynomials $R \colon \ell_2 \longrightarrow \ell_1$ and $Q \colon \ell_2 \longrightarrow \ell_2 \widehat{\otimes}_\pi^s \ell_2$ given by
$$R((\lambda_j)_j) = (\lambda_j^2)_j {\rm ~~and~~} Q(x) = x \otimes x. $$
In \cite[Example 2.10]{ryan.livro} it is proved that the operator
$$u \colon \ell_1 \longrightarrow  \ell_2 \widehat{\otimes}_\pi^s \ell_2~,~u((\lambda_j)_j) = \sum_{j}\lambda_j e_j \otimes e_j,$$
where $(e_j)_j$ are the canonical unit vectors, is an isometric isomorphism into (or, equivalently, a metric injection). The fact that $\ell_1$ is a Schur space guarantees that $u \in {\cal CC}(\ell_1; \ell_2 \widehat{\otimes}_\pi^s)$, hence $P : = u \circ R \in {\cal CC}\circ {\cal P}(^2 \ell_2; \ell_2 \widehat{\otimes}_\pi^s \ell_2)$. Since $\ell_2 \widehat{\otimes}_\pi^s \ell_2$ contains a (complemented) copy of $\ell_2$ (see \cite{fernandoblasco}), we know that
$\ell_2 \widehat{\otimes}_\pi^s \ell_2$ is not a Schur space, that is, $id_{\ell_2 \widehat{\otimes}_\pi^s \ell_2} = Q_L$ does not belong to $\cal CC$. By \cite[Proposition 3.2]{Geraldo.D.P} we conclude that $Q$ does not belong to ${\cal CC} \circ {\cal P}$. Moreover, for every $(\lambda_j)_j \in \ell_2$,
$$\|P((\lambda_j)_j)\| = \|u(R((\lambda_j)_j))\| = \|u((\lambda_j^2)_j)\| = \|(\lambda_j^2)_j\|_{\ell_1} = \|(\lambda_j^2)_j\|_{\ell_2}^2 = \|Q((\lambda_j)_j)\|. $$
So, $P$ belongs to ${\cal CC}\circ {\cal P}$, $\|P(x)\| = \|Q(x)\|$ for every $x$ but $Q$ does not belong to  ${\cal CC}\circ {\cal P}$, proving that  ${\cal CC}\circ {\cal P}$ fails the weak domination property. The example is complete because  ${\cal CC}\circ {\cal P}$ is an injective polynomial ideal by Corollary \ref{corol}.
\end{example}

\bigskip

\noindent Faculdade de Matem\'atica~~~~~~~~~~~~~~~~~~~~~~Departamento de Matem\'atica\\
Universidade Federal de Uberl\^andia~~~~~~~~ IMECC-UNICAMP\\
38.400-902 -- Uberl\^andia -- Brazil~~~~~~~~~~~~ 13.083-859 - Campinas -- Brazil\\
e-mail: botelho@ufu.br ~~~~~~~~~~~~~~~~~~~~~~~~~e-mail: leodan.ac.t@gmail.com

\end{document}